\documentclass [11pt,a4paper]{article}
\usepackage[T1]{fontenc}
\usepackage[ansinew]{inputenc}
\usepackage{amsthm,amsmath,amssymb}
\usepackage{amsfonts}
\usepackage{enumerate}

\newcommand{\R}{\mathbb R}
\newcommand{\s}{\mathbb S}

\newcommand{\eps}{\varepsilon}

\newtheorem{thm}{Theorem}[section]

\newtheorem{lem}{Lemma}[section]
\newtheorem{prop}{Proposition}[section]

\newtheorem{defn}{Definition}
\newtheorem{rem}{Remark}[section]

\begin{document}

\title{Weak geodesic flow on a semi-direct product and global solutions to the periodic Hunter-Saxton system}

\author{{\sc Marcus Wunsch}
\footnote{Research Institute for Mathematical Sciences,
 Kyoto University, Kyoto 606-8502, Japan, Email: \emph{mwunsch@kurims.kyoto-u.ac.jp}, \emph{marcus.wunsch@gmx.net}}}

\maketitle 

\begin{abstract}
\noindent
We give explicit solutions to the two-component Hunter-Saxton system on the unit circle. Moreover, we show how global weak solutions can be naturally constructed using the geometric interpretation of this system as a re-expression of the geodesic flow on the semi-direct product of a suitable subgroup of the diffeomorphism group of the circle with the space of smooth functions on the circle. These spatially and temporally periodic solutions turn out to be conservative. 
\end{abstract}



\noindent

\section{Introduction}
\setcounter{equation}{0}
In this paper, we are concerned with the two-component Hunter-Saxton system with periodic boundary conditions: 


\begin{equation} \label{2hs}
\begin{cases}
u_{txx} + u u_{xxx} + 2 u_x u_{xx} = \rho \rho_x, \qquad t > 0,\; x \in \s \simeq \R/\mathbb Z, \\
\rho_t + (u \rho)_x = 0 \\
\quad u(0,x) = \tilde u(x), \; \rho(0,x) = \tilde \rho(x). 
\end{cases}
\end{equation}
The Hunter-Saxton system \cite{Wun09,Wun10,wuwu,liyin1,liyin2} is a two-component generalization of the well-known Hunter-Saxton equation $u_{txx} + u u_{xxx} + 2 u_x u_{xx} = 0$ modeling the propagation of nonlinear orientation waves in a massive nematic liquid crystal (cf. \cite{h-s,bc,bhr,lenells,len08,len:weak,km,tiglay}), to which it reduces if $\tilde \rho$ is chosen to vanish identically. 

\par

In mathematical physics, the Hunter-Saxton system \eqref{2hs} arises as a model for the nonlinear dynamics of one-dimensional non-dissipative dark matter (the so-called Gurevich-Zybin system, see \cite{pavlov} and the references therein). 
Additionally, it is the short wave limit (using the scaling $(t,x) \mapsto (\eps t,\eps x)$, and letting $\eps \rightarrow 0$ in the resulting equations) of the two-component Camassa-Holm system originating in the Green-Naghdi equations which approximate the governing equations for water waves  \cite{ci,ELY07,guo,gz}. The Hunter-Saxton system is embedded in a more general family of coupled third-order systems \cite{Wun10} encompassing the axisymmetric Euler flow with swirl \cite{hl} and a vorticity model equation \cite{clm,osw}, among others (cf. \cite{pj,okamoto,w}). 

\par

Geometric aspects of \eqref{2hs} have recently been described in \cite{EE}: The Hunter-Saxton system can be realized as a geodesic equation on the semi-direct product of a subgroup of the group of circle diffeomorphisms with the space of smooth functions on the circle. This geometrical interpretation is closely linked to the re-expression of the Camassa-Holm and Degasperis-Procesi systems as (non-metric) Euler-Arnold equation (see, e.g., \cite{ekl10,ll}). \footnote{For the geometry of the Camassa-Holm {\it equation}, see, e.g., \cite{kour,misiolek,ck:geod,ck}.} 

\par

In this paper, we adopt the geometric viewpoint of \cite{EE} and adapt techniques of \cite{len:weak} developed for the Hunter-Saxton {\it equation} to construct global weak solutions to the two-component Hunter-Saxton system \eqref{2hs}. Weak solutions to the Hunter-Saxton system can, in analogy to the Hunter-Saxton equation \cite{hz,bc,len:weak}, be classified as either {\it dissipative} or {\it conservative}: While the energy $E(t) := \|u_x(t,.)\|_{L_2(\s)}^2 + \| \rho(t,.) \|_{L_2(\s)}^2$ is a decreasing functional for dissipative solutions, one has 
$$
E(t) = E(0) \qquad \mbox{ for almost every }t > 0
$$
for conservative solutions \cite{Wun10}. Here, we give a construction of conservative weak solutions with periodic boundary conditions.

\begin{rem}

Recently, the authors of \cite{liyin2} have proved that there are dissipative solutions to the $\mu$-Hunter-Saxton system which exist globally. The main differences to our results are that 
\begin{itemize}
\item we construct {\it conservative} weak solutions; 
\item we do not exclude zero values of the initial datum $\tilde \rho(x)$;  
\item our solutions are not only periodic in space, but also in time; 
\item our strategy involves the geometric interpretation of \eqref{2hs} (see \cite{EE}). 
\end{itemize}
\end{rem}

\noindent
The rest of this article is organized as follows: Section \ref{explicit} presents novel solution formulae of \eqref{2hs} and an explicit expression for the maximal time of existence. In Section \ref{semi}, we outline the analytic framework in which system \eqref{2hs} can be recast as a geodesic equation, following the approach of \cite{EE}. Finally, in Section \ref{main}, we show that the geodesic flow can be extended indefinitely in a weak sense, and we utilize this result in order to construct global conservative solutions to the Hunter-Saxton system. 

\begin{rem}[Notation]
The Hilbert spaces of functions (function equivalence classes) $f : \s \rightarrow \R$ which, together with their derivatives of order $k \ge 0$, are square-integrable, will be denoted by $H^k(\s)$, and the corresponding norm by $\| f \|_k$. If $k = 0$, we suppress the subscript and just write $\| f \|$. For the subspace of functions in $H^k(\s)$ with vanishing mean, we will use the symbol $H^k(\s) / \R$.
Lastly, the abbreviation $\{ f > 0\}$ for the set $\{ x \in \s:\; f(x) > 0 \}$ (and other analogous short forms) will be used throughout the text. 

\end{rem}

\section{Explicit solution formulae} \label{explicit}
\noindent
In this section, we provide new solution formulae for the Hunter-Saxton system. 
\\
Recall that the first component equation can be rewritten in terms of the gradient $u_x(t,x)$, 
\begin{equation} 
u_{tx} + u u_{xx} + \frac{1}{2} u_x^2 = \frac{1}{2} \rho^2 + a(t),  
\end{equation} 
\noindent
where the nonlocal term $a(t) = -\frac{1}{2} \int_{\s} \rho^2 + u_x^2 \;dx$ is determined by periodicity. 
It turns out (see \cite{Wun09,wuwu}) that $a'(t) = 0$, so that we may, without loss of generality\footnote{Observe that the Hunter-Saxton system is invariant under the scalings $u(t,x) \mapsto \alpha u(\alpha t,x)$, $\rho(t,x) = \alpha \rho(\alpha t,x)$, $\alpha \in \R$.}, set $a(t) = a(0) = -2$. \\
Let us introduce the Lagrangian flow map $\varphi$ solving 
\begin{equation} \label{flow}
\varphi_t(t,x) = u(t,\varphi(t,x)),\quad \varphi(0,x) = x \in \s.
\end{equation}
\noindent
In terms of $U(t,x) := u_x(t,\varphi)$, $\varrho(t,x) = \rho(t,\varphi)$, we rephrase  \eqref{2hs} as
\begin{equation} \label{lagran}
\begin{cases}
U_t + \frac{1}{2} U^2 = \frac{1}{2} \varrho^2 - 2\\
\varrho_t + U \varrho = 0 \\
\qquad U(0,x) =\tilde u_x(x),\; \varrho(0,x) = \tilde \rho(x). 
\end{cases}
\end{equation}
\noindent
Introducing the complex function $z(t,x) = U(t,x) + \sqrt{-1} \; \varrho(t,x)$, one sees after a short calculation that $z$ solves the Riccati-type differential equation
\begin{equation} \label{ric}
z_t(t,x) = -\frac{1}{2} \left( z(t,x)^2 + 4 \right), \quad z(0,x) = \tilde z(x) = \tilde u_x(x) + \sqrt{-1} \; \tilde \rho(x).
\end{equation}
\noindent 
Equation \eqref{ric} is explicitly solvable: 
$$
z(t,x) = \frac{2 \tilde z(x) - 4\tan t}{\tilde z(x) \tan t + 2}. 
$$
Separating real and imaginary parts, we get the following formulae for system \eqref{lagran}: 
\begin{eqnarray}
u_x(t,\varphi(t,x)) &=& \frac{4\cos (2t) \; \tilde u_x(x) + \sin (2t) \left[\tilde u_x(x)^2 + \tilde \rho(x)^2 - 4\right]}{\left[2 \cos t + \tilde u_x(x) \sin t \right]^2 + \tilde \rho(x)^2 \sin^2 t}, \label{ux} 
\\[0.2cm]
\rho(t,\varphi(t,x)) &=& \frac{4 \tilde \rho(x)}{\left[2 \cos t + \tilde u_x(x) \sin t \right]^2 + \tilde \rho(x)^2\sin^2 t}  \label{rho}. 
\end{eqnarray}
Interesting information can be drawn from the denominator in \eqref{ux}, \eqref{rho}: both solution components become unbounded at points $x^\ast \in \s$ which are both zeroes of $\tilde \rho$ and for which $\tilde u_x (x^\ast)= -2 \cot t$. 
The time-dependent and periodically recurring break-down set thus is the intersection
$$
B(t) = \{\tilde \rho= 0\} \cap \{ \tilde u_x = -2 \cot t \}. 
$$
Summing up, we have proven the subsequent theorem. 

\begin{thm}
Suppose $(\tilde u, \tilde \rho)^\dagger \in H^k(\s) \times H^{k - 1}(\s)$, $k \ge 3$ satisfies $\| \tilde u_x \|^2 + \| \tilde \rho \|^2 = 4$. 
Let 
$$\binom{u(t,.)}{\rho(t,.)} \in C([0,T^*); (H^k(\s)/\R) \times H^{k - 1}(\s)) \cap C^1([0,T^*); (H^{k - 1}(\s)/\R) \times H^{k - 2}(\s))$$ 
be the solution to the Hunter-Saxton system with initial data $(\tilde u,\tilde \rho)^\dagger$, and let $\varphi(t,x)$ with $\varphi(t,0) = 0$, $t \in [0,T^*)$, solve the Lagrangian flow map equation \eqref{flow}.
\par
Then 
\begin{equation} \label{phix}
\varphi (t,x) = \int_0^x \left\{ \left( \cos t + \frac{\tilde u_x(y)}{2} \; \sin t\right)^2 + \frac{\tilde \rho(y)^2}{4}\; \sin^2 t \right\}\;dy, 
\end{equation} 
and the first time when the diffeomorphism $\varphi(t,.): \; \s \rightarrow \s$ flattens out is given by  \hfill \\$T^\ast = \frac{\pi}{2} + \arctan \left\{ \frac{1}{2} \min_{\{\tilde \rho = 0\}} \tilde u_x(x) \right\}$. 
Lastly, the solution in terms of $\varphi$ reads 
\begin{equation*}
(u(t,.), \rho(t,.))^\dagger = \left(\varphi(t) \circ \varphi(t)^{-1}, \frac{\tilde \rho}{\varphi_x(t) \circ \varphi(t)^{-1}} \right)^\dagger. 
\end{equation*}
\end{thm}

\begin{rem}
From the expression for the maximal time of existence of solutions, we see that $T^*< \pi/2$ if there exist points $x$ where $\tilde \rho$ vanishes; else, one has $T^* = +\infty$, i.e., global existence in time (see \cite{wuwu}).   
The importance of the set $\{ \tilde \rho = 0 \}$ for the formation of singularities was described in \cite{Wun10}, Remark 5.2, for the case of (dissipative) solutions on the real line.
\end{rem}

\section{Semidirect products over the circle} \label{semi}
\noindent
In this section, we briefly outline the geometric approach of \cite{EE} to the two-component Hunter-Saxton system \eqref{2hs}. 
\par
Let {\sc Diff}$(\s)$ be the Lie group of the group of orientation-preserving diffeomorphisms of the circle, and denote by $C^\infty(\s)$ its Lie algebra. 
Also, define the following Fr\'echet Lie subgroup of {\sc Diff}$(\s)$ fixing one point: $$
\mbox{{\sc Diff}}_1(\s) := \{ \varphi \in \mbox{{\sc Diff}}(\s) :\; \varphi(1) = 1 \}. 
$$
Observe that the Lie algebra of {\sc Diff}$_1(\s)$ is the closed subspace 
$$
C_0^\infty(\s) := \{ f \in C^\infty(\s) : \; f(0) = 0 \}. 
$$
Now, we introduce the semidirect-product Lie group (cf. \cite{hmr,ht}) 
\begin{equation} \label{sdp}
C^\infty G = \mbox{{\sc Diff}}_1(\s) \; \textcircled{s} \; C^\infty(\s), 
\end{equation}
with the group multiplication $(\varphi,f) \; (\psi,g) := (\varphi \circ \psi, g + f \circ \psi)$, where $(\varphi,f)$, $(\psi,g)$ are two given elements of $G$, and $\circ$ stands for composition. The neutral element of $G$ is $(id,0)$, and the inverse of $(\varphi,f)$ is $(\varphi^{-1}, - f \circ \varphi^{-1})$. On the Lie algebra $T_{id} C^\infty G \simeq C_0^\infty(\s) \times C^\infty(\s)$, the Lie bracket is given by 
$$
[(u,\rho),(v,\sigma)] = (u_x v - u v_x, u\sigma_x - v \rho_x), \qquad (u,\rho),\;(v,\sigma) \in C^\infty \frak{g} := C_0^\infty(\s) \times C^\infty(\s). 
$$
The semi-direct product $C^\infty G$ can be equipped with a Riemannian metric via the inertia operator $\mathbb A : C^\infty \frak{g} \rightarrow C^\infty \frak{g}, \; (u,\rho)^\dagger \mapsto (-u_{xx},\rho)^\dagger$ by setting 
\begin{equation}
\left \langle \binom{u}{\rho},\binom{v}{\sigma} \right \rangle_{(id,0)}^\mathbb A := \frac{1}{4} \int_\s \left( (u,\rho)^\dagger \big| \mathbb A(v,\sigma)^\dagger \right)_{\R^2} \;dx := \frac{1}{4} \int_\s u_x v_x + \rho \sigma \;dx
\end{equation}
for $(u,\rho)^\dagger$, $(v,\sigma)^\dagger \in C^\infty \frak{g}$, and by extending it to all of $C^\infty G$ by right translation \cite{ekl10}. 
The topology induced by this Hilbert structure is weaker than the original Fr\'echet topology, which is why such a structure on $C^\infty G$ is called a {\it weak} Riemannian metric. 
Observe also that $\mathbb A$ gives rise to the norm 
$$
\left \|\left| (u,\rho)^\dagger \right|\right \|_\mathbb A := \frac{1}{2} \sqrt{ \| u_x \|^2 + \| \rho \|^2}. 
$$
This analytic framework facilitates the following geometric interpretation of the Hunter-Saxton system \eqref{2hs}: 
\begin{prop}[\cite{EE}, Corollary 3.2] \label{escher}
The system \eqref{2hs} describes the geodesic flow of the right-invariant $\dot H^{1}(\s) \times H^0(\s)$ metric on the Fr\'echet Lie group {\sc Diff}$_1(\s) \; \textcircled{s} \; C^\infty(\s)$. 
\end{prop}



\subsection*{The space $\mathcal M_{AC}^0$}
For the {\it weak formulation} of the geodesic flow of the right-invariant $\dot H^1(\s) \times H^0(\s)$ metric, we define the space $\mathcal M^0_{AC}$ as the semi-direct product $M^{AC} \textcircled{s} H^0(\s)$, where $M^{AC}$ is the set of nondecreasing absolutely continuous functions $\varphi: [0,1] \rightarrow [0,1]$ with $\varphi(0) = 0$ and $\varphi(1) = 1$ (see \cite{len:weak}). 
The tangent space at the identity is naturally defined as 
$$
T_{(id,0)} \; \mathcal M^0_{AC} := \{ u \in H^1(\s) : \; u(0) = 0\} \times H^0(\s). 
$$

\subsubsection*{Characterization of tangent spaces.}
In analogy with \cite{len:weak}, we have the subsequent lemma. 
\begin{lem} \label{character}
Let $\varphi \in AC(\s)$, the set of absolutely continuous functions $\s \rightarrow \R$, and write 
\begin{equation} \label{null}
N := \{ x \in \s:\; \varphi_x(x) \mbox{ exists and equals } 0 \}. 
\end{equation}
Also, let $\varrho \in H^0(\s)$. Then we have the characterization 
\begin{eqnarray} \label{tM}
&&T_{ ( \varphi, \varrho) } \; \mathcal M^0_{AC} = 
\left \{ (U,F)^\dagger \in AC(\s) \times H^0(\s) :  \; U(0) = 0, \right. \\
&&Ê\left.  \qquad \qquad \quad \quad \;
U_x = 0 \; \mbox{ a.e. on } N \mbox{ and } \int_{ \s \setminus N} \frac{U_x^2}{\varphi_x} + (F - \varrho)^2 \varphi_x\; dx < + \infty \nonumber
 \right \}. 
\end{eqnarray}
\noindent
Furthermore, for any elements $(U,F)$, $(V,G) \in T_{(\varphi,\varrho)} \; \mathcal M^0_{AC}$, the inner product reads 
\begin{equation} \label{inprod}
\left \langle (U,F), (V,G) \right \rangle _{(\varphi,\varrho)}^\mathbb A = \frac{1}{4} \int_{\s \setminus N} \frac{U_x V_x}{\varphi_x} + (F - \varrho)(G - \varrho)\; \varphi_x \;dx. 
\end{equation}
\end{lem}


\begin{proof}
Define the set $A := \{ x \in \s : \; \varphi_x(x) \mbox{ exists and } \varphi_x(x) > 0 \}$. The unit circle $\s$ is thus the union of the three sets $A$, $N$, and $Z$, the latter being of measure zero. Moreover, the sets $\varphi(A)$, $\varphi(N)$, and $\varphi(Z)$ are measurable, where $\varphi(A)$ has full Lebesgue measure $\lambda$, while for the other two one has $\lambda( \varphi(N)) = \lambda ( \varphi(Z) ) = 0$.    
\par 
Suppose that $(U,F) \in T_{(\varphi,\varrho)} \mathcal M^0_{AC}$ for some $(\varphi,\varrho) \in \mathcal M^0_{AC}$. By definition, this means that there exist $u \in H^1(\s)$ with $u(0) = 0$ such that $U = u \circ \varphi$, and that there is a $\rho \in H^0(\s)$ such that $F = \varrho + \rho \circ \varphi$. One can see that $U(0) = 0$ and that $U_x = 0$ a.e. on $\s$ (cf. \cite{len:weak}). 
Also, 
\begin{eqnarray} \label{norm}
\int_A \frac{U_x^2}{\varphi_x} + (F - \varrho)^2\; \varphi_x\;dx &=& 
\int_A \left[ (u_x \circ \varphi)^2 + (\rho \circ \varphi)^2 \right] \varphi_x \;dx \\ 
&=& \int_\s \left[ (u_x \circ \varphi)^2 + (\rho \circ \varphi)^2 \right] \varphi_x \;dx 
= \int_\s u_x^2 + \rho^2\;dy < +\infty. \nonumber
\end{eqnarray} 
This proves the inclusion $\subset$ in \eqref{tM}. 
\par 
Conversely, for $U \in AC(\s)$ satisfying $U(0) = 0$, $U_x = 0$ a.e. on $N$ and $\int_{\s \setminus N}\frac{U_x^2}{\varphi_x} \;dx < \infty$, one can show as in \cite{len:weak} that there exists a $u \in H^0(\s)$ with $u(0) = 0$. As for the second component, we define a function $\rho\;$ a.e. on $\s$ by 
$$
\rho(y) = (F - \varrho)(\varphi^{-1}(y)), \quad y \in \varphi(A). 
$$
Now we see that
\begin{eqnarray*}
\int_\s \rho^2\;dy &=& \int_\s \left[ (F - \varrho)\circ \varphi^{-1} \right]^2 \;dy = \int_\s \left[ (F - \varrho) \circ \varphi^{-1} \right]^2 \circ \varphi \;\varphi_x \;dy,
\end{eqnarray*}
which equals to $\int_A (F - \varrho)^2 \;dx$ and is finite by Young's inequality, since both $F$ and $\varrho$ are square-integrable. This shows the inclusion $\supset$ in \eqref{tM}. 
We finish the proof by noticing that \eqref{inprod} is a consequence of \eqref{norm}. 
\end{proof}

\section{Main results} \label{main}
\noindent
Before stating our main results, let us introduce the Christoffel operator $\Gamma_{(\varphi,\varrho)}(.,.)$ defined in \cite{ekl10} on the product $H^k(\s) \times H^{k - 1}(\s)$, $k > 5/2$, as 
\begin{equation*}
\Gamma_{(id,0)} \left( (u,\rho), (v,\sigma) \right) = \binom{\Gamma^0_{id} (u,v) - \frac{1}{2} A^{-1} \partial_x (\rho \sigma)}{-\frac{1}{2} (u_x \sigma + v_x \rho)}, 
\end{equation*}
where $A = -\partial_x^2$. This operator is then extended by right translation to the tangent spaces of the semi-direct product $C^\infty G$ (cf. \cite{EE}). Recall that the inverse of $A$ is given by  
$$
(A^{-1} f) \; (x) = - \int_0^x \int_0^y f(z) \;dz \;dy + x \int_0^1 \int_0^y f(z) \;dz \;dy
$$
for all $f \in H^k(\s)/\R$, $k \ge 1$ (see, e.g., \cite{len:weak}).

\noindent
The following statement asserts the global existence of a weak geodesic flow on $\mathcal M^0_{AC}$. 

\begin{thm} \label{wgf}
Let $(\tilde u,\tilde \rho) \in T_{(id,0)} \mathcal M^0_{AC} = \{ u \in H^1(\s) : \; u (0) = 0 \} \times H^0(\s)$ with $\| \tilde u_x \|^2 + \| \tilde \rho \|^2 = 4$. 
Define 
\begin{equation} \label{varphi}
\varphi(t,x) := \int_0^x \left( f^2(t,y) + g^2(t,y) \right) \;dy, \mbox{ where }
\begin{cases}
\; f(t,x) := \left( \cos t + \frac{\tilde u_x}{2} \sin t \right), \mbox{ and } \\[0.1cm]
\; g(t,x) := \frac{\tilde \rho}{2} \sin t, \quad (t,x) \in [0,\infty) \times \s. 
\end{cases}
\end{equation}
Moreover, let 
\begin{equation} \label{varrho}
\varrho(t,x) := \tilde \rho(x)  \int_0^t \frac{\chi_{ \{\varphi_x(s,.) > 0 \} }}{\varphi_x(s,x)} \;ds. 
\end{equation}
\\
Then the following statements are true. 
\begin{enumerate}[(i)]
\item For each time $t \ge 0$, $( \varphi(t,.), \varrho(t,.)) \in \mathcal M_{AC}^0$. 
\item For each time $t \ge 0$, $( \varphi_t(t,.), \varrho_t(t,.)) \in T_{(\varphi,0 )} \; \mathcal M^0_{AC}$. 
\item If $\lambda$ denotes Lebesgue measure on $\s$, then, for $t \ge 0$,  
\begin{equation} \label{lebesgue}
\left\langle \binom{\varphi_t}{\varrho_t}, \binom{\varphi_t}{\varrho_t} \right\rangle_{(id,0)}^\mathbb A = 
\begin{cases}
1,  \quad \qquad \qquad\; \quad \; \qquad \; \qquad \quad \; \quad \; \; \; \qquad\sin t = 0 \\
1 - \frac{1}{\sin^2 t} \lambda \left(\{ \tilde u_x = -2 \cot t \} \cap \{ \tilde \rho = 0\} \right),  \; \sin t \neq 0. 
\end{cases}
\end{equation}
Specifically, the geodesic has constant energy almost everywhere: 
\begin{equation} \label{const:ener}
\left \langle \binom{\varphi_t}{\varrho_t}, \binom{\varphi_t}{\varrho_t} \right \rangle_{(id,0)}^\mathbb A = \left \langle \binom{\varphi_t(0)}{\varrho_t(0)}, \binom{\varphi_t(0)}{\varrho_t(0)} \right \rangle_{(id,0)}^\mathbb A \quad \mbox{ for a.e. }t \in [0,\infty). 
\end{equation}
\item The geodesic equation holds in the weak form 
\begin{equation} \label{geodesic}
\binom{\varphi_{tt}}{\varrho_{tt}} = \Gamma_{(\varphi,\varrho)} \left( (\varphi_t,\varrho_t), (\varphi_t,\varrho_t) \right) \quad \mbox{ for a.e. } t \in [0,\infty).
\end{equation}
\end{enumerate}
\end{thm}

\noindent
The weak geodesic formulation \eqref{geodesic} allows us to study weak solutions to the Hunter-Saxton system \eqref{2hs}. 

\begin{defn}[Definition of a weak solution to \eqref{2hs}] \label{we:so}
The pair $(u,\rho)^\dagger: \left( [0,\infty) \times \s \right) ^2 \rightarrow \R$ is a global conservative solution of equation \eqref{2hs} with initial data $(\tilde u,\tilde \rho)^\dagger \in H^1 \times H^0(\s)$ if 
\begin{enumerate}[(a)]
\item for each $t \in [0,\infty)$, the map $x \mapsto u(t,x)$ is in $H^1(\s)$;
\item $u \in C([0,\infty) \times \s; \R)$ and $u(0,.) = \tilde u$ point-wise on $\s$; $\rho(0,.) = \tilde \rho\;$ a.e. on $\s$; 
\item the maps $t \mapsto u_x(t,.)$ and $t \mapsto \rho(t,.)$ belong to the space $L^\infty([0,\infty); H^0(\s))$; 
\item the map $t \mapsto u(t,.)$ is absolutely continuous from $[0,\infty)$ to $H^0(\s)$ and satisfies 
$$
u_t + u u_x = \frac{1}{2} \left\{ \int_0^x \left( u_x^2 + \rho^2 \right) \;dy +  x \int_\s \left( u_x^2 + \rho^2 \right) \;dy \right\} \; \mbox{ in } H^0(\s) \mbox{ for a.e. } t \in [0,\infty); $$
and the map $t \mapsto \rho(t,.)$ satisfies 
$$
\rho_t + (u\rho)_x = 0 \quad \mbox{ in } H^0 \mbox{ for a.e. } t \in [0,\infty). 
$$
\end{enumerate} 
\end{defn}

\noindent
The next theorem asserts that there are global conservative solutions to the Hunter-Saxton system \eqref{2hs}. 

\begin{thm}[Global weak solutions to the periodic Hunter-Saxton system] \label{gws}
For any initial data $(\tilde u,\tilde \rho)^\dagger \in H^1(\s) \times H^0(\s)$ satisfying $\| \tilde u_x \|^2 + \| \tilde \rho \|^2 = 4$ and $\tilde u(0) = 0$.
Then the pair
\begin{equation} \label{weak:sol}
\binom{u(t,\varphi(t,x))}{\rho(t,\varphi(t,x))} := \binom{\varphi_t(t,x)}{\varrho_t(t,x)}, \quad (t,x) \in [0,\infty) \times \s,  
\end{equation}
constitutes a global weak solution of the Hunter-Saxton system \eqref{2hs} with initial data $(\tilde u,\tilde \rho)^\dagger$. It furthermore holds that, for $t \ge 0$, 
\begin{equation} \label{conservative}
\left \| \left | (u(t,.),\rho(t,.))^\dagger \right\| \right |_\mathbb A^2 = 
\begin{cases}
\left \| \left | (\tilde u, \tilde \rho)^\dagger \right\| \right |_\mathbb A^2 \qquad \qquad \qquad \qquad \qquad \qquad\qquad \quad \;\;\;\; \;  \sin t = 0\\
\left \| \left | (\tilde u, \tilde \rho)^\dagger \right\| \right |_\mathbb A^2 - \frac{1}{\sin^2 t} \lambda \left(\{ \tilde u_x = -2 \cot t \} \cap \{ \tilde \rho = 0\} \right),  \; \sin t \neq 0. 
\end{cases}
\end{equation}
This solution is conservative in the sense that 
$$
\| u_x(t,.) \|^2 + \| \rho(t,.) \|^2 = \| \tilde u_x \|^2 + \| \tilde \rho \|^2 \quad \mbox{ for a.e. } t \in [0,\infty).  
$$ 
\end{thm}

\subsection*{Proof of the main results}

\begin{proof}[Proof of Theorem \ref{wgf}, (i).]
The result follows in view of the definition of $\varphi$ (equation \eqref{phix}) and of the condition $\| \tilde u_x \|^2 + \| \tilde \rho \|^2 = 4$. 
\end{proof}

\begin{proof}[Proof of Theorem \ref{wgf}, (ii).]
By Lemma \ref{character}, we have to verify the following five conditions: 
\begin{enumerate}[1.]
\item $\varphi_t \in AC(\s)$, $\varrho_t \in H^0(\s)$;
\item $\varphi_t(t,0) = 0$; 
\item $\varphi_{tx} = 0\;$ a.e. on $N$; and 
\item $\int_{\s\setminus N} \dfrac{\varphi_{tx}^2}{\varphi_x} + \varrho_t^2 \; \varphi_x \;dx < + \infty$. 
\end{enumerate}
Condition $4.$ is a consequence of the computations in the proof of $(iii)$ below. The map $x \mapsto \varphi_t(t,x)$ is absolutely continuous, $\varrho_t \in H^0(\s)$ by \eqref{varrho}, and $\varphi_t(t,0) = \varphi_t(t,1) = 0$. Moreover, we see  that $\varphi_{tx} = (u_x \circ \varphi) \; \varphi_x = 0$ 
whenever $\varphi_x = 0$. This proves Conditions 1.--3. 
\end{proof}

\begin{proof}[Proof of Theorem \ref{wgf}, (iii).]
Elementary but slightly tedious calculations reveal that  
\begin{eqnarray} 
\int_{ \{ \varphi_x > 0 \} } \left( \frac{\varphi_{tx}^2}{\varphi_x} + \varrho_{t}^2\;\varphi_x \right) \;dx &=& 4 \int_{ \{ \varphi_x > 0 \} } \left( f_t^2 + g_t^2 \right) \;dx \label{reveal} \\
&=& 4 \int_\s \left( f_t^2 + g_t^2 \right) \;dx - 4 \int_{\{ f = 0 \} \cap \{ g = 0 \}} \left( f_t^2 + g_t^2 \right) \;dx \nonumber \\
&=& 4 - 4 \int_{\{ f = 0 \} \cap \{ g = 0 \}} \left( f_t^2 + g_t^2 \right) \;dx, \label{4}
\end{eqnarray}
where we used that 
$$
4 \int_\s \left( f_t^2 + g_t^2 \right) \;dx = 4 \sin^2 t + \cos^2 t \int_\s \left( \tilde u_x^2 + \tilde \rho^2 \right) \;dx = 4
$$
to derive equality in \eqref{4}. The set $\{ f = 0 \} \cap \{ g = 0 \}$ is empty whenever $\sin t = 0$, whence
\begin{equation}
\left\langle \binom{\varphi_t}{\varrho_t}, \binom{\varphi_t}{\varrho_t} \right\rangle_{(\varphi,\varrho)}^\mathbb A = 1, \qquad t \ge 0,\; \sin t = 0. 
\end{equation}
Whenever $\sin t \neq 0$, then 
\begin{eqnarray*}
&& 4 - 4 \int_{\{ f = 0\} \cap \{ g = 0 \}} \left(\sin^2 t - \tilde u_x \cos t \sin t + \frac{\tilde u_x^2 + \tilde \rho^2}{4} \cos^2 t \right) \;dx \\[0.2cm]
&=& 4 - \frac{4}{\sin^2 t} \lambda \left(\{ \tilde u_x = - 2 \cot t \} \cap \{ \tilde \rho = 0 \}\right).  
\end{eqnarray*}
Summing up, we gain 
\begin{equation}
\int_{ \{ \varphi_x > 0 \} } \left( \frac{\varphi_{tx}^2}{\varphi_x} + \varrho_t^2\; \varphi_x \right) \;dx = 
\begin{cases}
4 \qquad \qquad \qquad \quad \qquad \qquad \qquad \qquad \qquad t \ge 0, \; \sin t = 0 \\
4 - \frac{4}{\sin^2 t} \lambda \left(\{ \tilde u_x = -2 \cot t \} \cap \{ \tilde \rho = 0\} \right) \quad t \ge 0, \; \sin t \neq 0.  
\end{cases}
\end{equation}
This proves the identity \eqref{lebesgue}. In order to complete the proof of (iii), it remains to show that \eqref{lebesgue} implies \eqref{const:ener}. This follows once we realize that $\lambda(\{ f = 0 \} \cap \{ g = 0 \}) = 0$. But it has already been demonstrated in \cite{len:weak} that $\{ f = 0 \}$ is a set of measure zero, which, in view of the measurability of the map $(t,x) \mapsto g(t,x) : [0,\infty) \times \s \rightarrow \R$, proves \eqref{const:ener}.     
\end{proof}


\begin{proof}[Proof of Theorem \ref{wgf}, (iv).]
Define 
$$
\phi(t,x) := \int_0^x f^2(t,y) \;dy, \quad \psi(t,x) := \int_0^x g^2(t,y) \; dy, 
$$
so that $\varphi(t,x) = \phi(t,x) + \psi(t,x)
$ (see equation \eqref{phix}).
Then 
$$
\varphi_t(t,x) = 2 \int_0^x f f_t +  g g_t\;dy, \quad \varphi_{tt} =  2 \int_0^x f_t^2 + g_t^2 + f f_{tt} + g g_{tt} \;dy.
$$ 
Since $f_{tt} = -f$ and $g_{tt} = -g$, the second-order derivative of $\varphi$ can be rewritten as 
$$
\varphi_{tt}(t,x) = 2 \int_0^x \left( f_t^2 + g_t^2\right) \; dy - 2 \int_0^x \left( f^2 + g^2 \right) \;dy. 
$$
This integral is split into two terms: 
\begin{equation} \label{first}
2 \int_0^x \chi_{ \{ \varphi_x > 0 \} } (f^2 + g^2) \left( \frac{f_t^2 + g_t^2}{f^2 + g^2} - 1 \right) \;dx + 2 \int_0^x \chi_{\{ \varphi_x = 0 \}} \left( f_t^2 + g_t^2 \right) \;dx. 
\end{equation}
From equation \eqref{reveal}, we deduce that $f_t^2 + g_t^2 = \frac{1}{4} \frac{\varphi_{tx}^2 + \tilde \rho^2}{\varphi_x}$ on $\{ \varphi_x > 0\}$, so that, for the first term in \eqref{first}, we may compute 
\begin{eqnarray*}
&& 2 \int_0^x \chi_{\{ \varphi_x > 0 \}} (f^2 + g^2) \left( \frac{f_t^2 + g_t^2}{f^2 + g^2} - 1 \right) \;dy 
\\
&&\stackrel{\eqref{reveal}}{=} \frac{1}{2} \int_0^x \chi_{\{ \varphi_x > 0 \}} \varphi_x \left[ (u_x\circ \varphi)^2 + (\rho\circ \varphi)^2 - 4 \right]\; dy 
\\
&&= 
\frac{1}{2} \int_0^{\varphi(x)} \chi_{ \varphi(t,\{ \varphi_x > 0 \}) } \left( u_x^2 + \rho^2 - 4 \right) \;dy = \frac{1}{2} \int_0^{\varphi(x)} (u_x^2 + \rho^2)\;dy - 2 \varphi(x).
\end{eqnarray*}
Since the second term in \eqref{first} vanishes identically on $\s$ for a.e. $t \in [0,\infty)$ (cf. \cite{len:weak}), we infer (recalling that $\| u_x(t,.) \|^2 + \| \rho(t,.) \|^2 \equiv 4$)
$$
\varphi_{tt} = \Gamma_{(\varphi,\varrho)}^{(1)} \left( (\varphi_t,\varrho_t), (\varphi_t,\varrho_t) \right)  \quad \mbox{ a.e. in time}.
$$
As $\varrho_t = \rho \circ \varphi$,  one immediately sees that $\varrho_{tt} = - (u_x \rho) \circ \varphi$, and so
$$
\varrho_{tt} = \Gamma_{(\varphi,\varrho)}^{(2)} \left( (\varphi_t,\varrho_t), (\varphi_t,\varrho_t) \right)  \quad \mbox{ a.e. in time}.
$$
This finishes the proof of Theorem \ref{wgf}. 
\end{proof}

\begin{proof}[Proof of Theorem \ref{gws}, Condition (a)]
The first condition of Definition \ref{we:so} is fulfilled as a consequence of Theorem \ref{wgf}(ii) and Lemma \ref{character}. In view of \eqref{const:ener} and the identity 
$$
\left\| \left| \binom{u(t,.)}{\rho(t,.)} \right\| \right|_\mathbb A^2 = \left \langle \binom{\varphi_t}{\varrho_t}, \binom{\varphi_t}{\varrho_t} \right \rangle_{(id,0)}^\mathbb A, \qquad t \ge 0,
$$
we moreover see that equation \eqref{conservative} holds. 
\end{proof}

\begin{proof}[Proof of Theorem \ref{gws}, Condition (b)]
The identity $\rho(0,.) = \tilde \rho\;$ a.e. on $\s$ follows directly from the definition of $\varrho$ \eqref{varrho}, and the continuity of the map $(t,x) \mapsto u(t,x) : [0,\infty) \times \s \rightarrow \R$ is a result of arguments akin to those of \cite{len:weak}. 
\end{proof}

\begin{proof}[Proof of Theorem \ref{gws}, Condition (c)] 
That the map $t \mapsto u_x(t,.)$ belongs to the space $L^\infty ([0,\infty); H^0(\s))$ can be proved, with the proper adaptations, as in \cite{len:weak}. Concerning the map $t\mapsto \rho(t,.)$, we first notice that 
\begin{eqnarray*}
\int_\s \rho(t,y) \theta(y)\;dy &=& \int_\s \rho(t,\varphi(t,x)) \; \theta(\varphi(t,x)) \; \varphi_x(t,x) \;dx \\
&=& \int_\s \tilde \rho(x) \; \theta(\varphi(t,x)) \;dx,   
\end{eqnarray*} 
which shows that  
\begin{equation} \label{trho}
t \mapsto \int_\s \rho(t,y) \; \theta(y) \;dy 
\end{equation}
is continuous for each $\theta \in C^\infty(\s)$. For a general $\theta \in H^0(\s)$, we choose an approximating sequence $\{ \theta_n \}_n \subseteq C^\infty(\s)$. Since strong convergence implies weak convergence, it follows, for every fixed $t$, that
$$
\int_\s \rho(t,y) \theta_n(y)\; dy \; \stackrel{n \rightarrow \infty}{\longrightarrow} \; \int_\s \rho(t,y) \theta(y)\;dy. 
$$  
This implies, for any $\theta \in H^0(\s)$, that the map \eqref{trho} is a point-wise limit of measurable maps and thus itself measurable. We finish the argument by invoking Pettis' theorem. 
\end{proof}

\begin{proof}[Proof of Theorem \ref{gws}, Condition (d)]
We make two claims. \\[0.1cm]
\textbf{Claim 1.} For any test functions $\eta \in C^\infty_c((0,\infty))$ and $\vartheta \in C^\infty(\s)$, it holds that
\begin{equation} \label{claim1}
\int_\s \left\{ \int_0^\infty u(t,y) \eta_t(t)\;dt + \int_0^\infty \left( \Gamma_{(id,0)}^{(1)} (( u,\rho ), ( u,\rho )) - u u_x\right) (t,y) \eta(t) \;dt \right\} \vartheta(y) \;dy = 0. 
\end{equation}
\textbf{Claim 2.} Moreover, for any test functions $\eta \in C^\infty_c((0,\infty))$ and $\vartheta \in C^\infty(\s)$, one has that 
\begin{equation} \label{claim2}
\int_\s \left\{ \int_0^\infty \rho(t,y) \eta_t(t) \;dt - \int_0^\infty \left( \Gamma_{(id,0)}^{(2)} ((u,\rho),(u,\rho)) - u \rho_x \right) \eta(t) \; dt \right\} \vartheta(y) \;dy = 0. 
\end{equation}
For the moment, let us assume that Claim 1 is true. Then 
$$
\int_0^\infty u(t,y) \eta_t(t) \;dt = -\int_0^\infty \left( \Gamma_{(id,0)}^{(1)} ((u,\rho),(u,\rho)) - u u_x \right) (t,y) \eta(t) \;dt
$$
for any $\eta \in C^\infty_c((0,\infty))$, so that we may conclude that the map $t \mapsto u(t,.), \; [0,\infty) \rightarrow H^0(\s)$ is absolutely continuous and that the equality 
$$
u(t) = \tilde u + \int_0^t \left( \Gamma_{(id,0)}^{(1)} ((u,\rho),(u,\rho)) - u u_x \right) (s) \;ds
$$
holds in $H^0(\s)$. 
Since
$$
\Gamma_{(id,0)}^{(1)} ((u,\rho),(u,\rho)) = \left( x \mapsto \frac{1}{2} \int_0^x u_x^2 + \rho^2 \;dy - \frac{x}{2} \int_\s u_x^2 + \rho^2 \;dy \right)
$$
this proves that $u$ is a weak solution to the first component of the Hunter-Saxton system \eqref{2hs}. \\
As for the second component, \eqref{claim2} yields 
$$
\int_0^\infty \rho(t,y) \; \eta_t \;dt = \int_0^\infty  \Gamma_{(id,0)}^{(2)} ((u,\rho),(u,\rho))  (t,y) \; \eta(t) \;dt, 
$$
for any test function $\eta \in C^\infty_c((0,\infty))$, whence $t \mapsto \rho(t,.)$, $[0,\infty) \rightarrow H^0(\s)$ solves 
$$
\rho(t) = \int_0^t  \Gamma_{(id,0)}^{(2)} ((u,\rho),(u,\rho)) \; ds. 
$$
In view of 
$$
\Gamma_{(id,0)}^{(2)} ((u,\rho),(u,\rho)) = \left( x \mapsto -(u\rho)_x \right), 
$$
this shows that $\rho$ solves the second component of the Hunter-Saxton system \eqref{2hs}. The proof is thus complete. 
\end{proof}

\begin{proof}[Proof of Claim 1.]
This proof follows, mutatis mutandis, the lines of \cite{len:weak} (p. 655), so we omit it here. 
\end{proof}

\begin{proof}[Proof of Claim 2.]
We study the spatial integrand in \eqref{claim2} first:  
\begin{eqnarray*}
\int_\s \rho(t,y) \;\vartheta(y)\; dy &=& \int_\s \rho(t,\varphi(t,x))\;\vartheta(\varphi(t,x))\; \varphi_x(t,x) \;dx \\
&=& \int_\s \varrho_t (t,x) \; \vartheta(\varphi(t,x)) \; \varphi_x(t,x) \;dx. 
\end{eqnarray*}
Integrating by parts in \eqref{claim2} now gives
\begin{eqnarray*}
&&\int_0^\infty \int_\s \rho(t,y)\; \vartheta(y) \;dy \; \eta_t(t) \;dt 
\\[0.2cm]
&=& - \int_0^\infty \int_\s \frac{d}{dt} \left( \varrho_t(t,x)\; \vartheta(\varphi(t,x))\; \varphi_x(t,x) \right) \; \eta(t) \;dt 
\\[0.2cm]
&=& - \int_0^\infty \int_\s \left\{ \varrho_{tt}(t,x) \; \vartheta(\varphi(t,x)) \varphi_x(t,x) + \varrho_t(t,x) \; \varphi_t(t,x) \; \vartheta_x(\varphi(t,x)) \; \varphi_x(t,x) \right.
\\[0.2cm]
&&-  \left. \varrho_t(t,x)\;\vartheta(\varphi(t,x)) \varphi_{tx} (t,x) \right\} \;dx \; \eta(t) \; dt.
\end{eqnarray*}
If we employ the relation $\varrho_{tt} = \Gamma_{(id,0)}^{(2)} ((\varphi_t,\varrho_t),(\varphi_t,\varrho_t))$ (cf. \eqref{geodesic}), the change of variables $y = \varphi(t,x)$ yields 
\begin{eqnarray*}
&&-\int_0^\infty \int_\s \left\{  \Gamma_{(id,0)}^{(2)} ((u,\rho),(u,\rho))(t,y) \vartheta(y) \right. 
\\[0.2cm]
&&+ \left. \rho(t,y)\; u(t,y) \; \vartheta_x(y) + \rho(t,y) \vartheta(y) u_x(t,y) \right\} dy \eta(t) dt 
\\[0.2cm]
&=& -\int_0^\infty \int_\s \left(  \Gamma_{(id,0)}^{(2)} ((u,\rho),(u,\rho))(t,y) -  u(t,y)\;\rho_x(t,y)  \right) \; \vartheta(y) \;dy\; \eta(t) \;dt. 
\end{eqnarray*}
\end{proof}

\section*{Acknowledgments}
The author was partially supported by JSPS Postdoctoral Fellowship P09024. He is obliged to {\sc J.~Escher} for the personal communication \cite{EE} and for beneficial discussions.

 \end{document}